\documentclass[11pt]{article}
\pdfoutput=1

\usepackage[margin=1.2in]{geometry}

\usepackage{graphicx}
\usepackage{amssymb,amsmath}
\usepackage{amsthm,enumerate}
\usepackage{hyperref,xcolor}
\usepackage{mathrsfs}
\usepackage{extarrows}
\usepackage{textcomp}
\allowdisplaybreaks[4]

\makeatletter

\newdimen\bibspace
\makeatother

\numberwithin{equation}{section}

\newtheorem{theorem}{Theorem}[section]
\newtheorem{lemma}[theorem]{Lemma}
\newtheorem{proposition}[theorem]{Proposition}

\newtheorem{remark}[theorem]{Remark}

\def\<{\langle}
\def\>{\rangle}

\def\XXint#1#2#3{{\setbox0=\hbox{$#1{#2#3}{\int}$ }
\vcenter{\hbox{$#2#3$ }}\kern-.6\wd0}}

\begin{document}

\title{Entire solutions to the parabolic Monge--Amp\`ere equation\\ with unbounded nonlinear growth in time}

\author{Ning An, Jiguang Bao, Zixiao Liu\footnote{J. Bao was supported by the National Key Research and Development Program of China (No. 2020YFA0712904) and the Beijing Natural Science Foundation (No. 1222017). Z. Liu was supported by the China Postdoctoral Science Foundation (No. 2022M720327).}}
\date{\today}

\maketitle

\begin{abstract}
The Liouville type theorem on the parabolic Monge--Amp\`ere equation $-u_t\det D^2u=1$  states that any entire parabolically convex classical solution  must be of form $-t+|x|^2/2$ up to a re-scaling and transformation, under additional assumption that partial derivative with respect to time variable $u_t$ is strictly negative and bounded. In this paper, we study the case when $u_t$ is unbounded, prove an existence result of entire parabolically convex smooth solution and investigate the asymptotic behavior near infinity.
 \\[1mm]
\text{\hspace{4.7em} Existence}, Asymptotic behavior.
\\[1mm]
\end{abstract}

\section{Introduction}

As a family of second order fully nonlinear elliptic equations, the Monge--Amp\`ere equation plays an important role in the Weyl and Minkowski problem \cite{Nirenberg-WeylnMinkovskiProb,Li-Sheng-Wang-MinkowskiProb-MA}, optimal transport problem \cite{Kantorovitch-OptimalTrans-MA,DePhilippis-Figalli-OptimalTrans-MA}, computational geometry \cite{Gu-Luo-Sun-Yau-ComputGeom-MA}, mean curvature equations of gradient graphs \cite{Chong-Rongli-Bao-SecondBoundary-SPL,Liu-Bao-2021-Expansion-LagMeanC}, etc. In 1976, Krylov \cite{Krylov-ParaboMA-Operator} first introduced the following type of parabolic Monge--Amp\`ere equations
\[
-u_t\cdot \det D^2u=f(x,t,u,Du)
\]
and established an Alexandroff--Bakelman--Pucci type maximum principle for parabolic equation. Hereinafter, $u_t$ denotes the derivative of $u$ with respect to time variable $t$, $Du$ denotes the gradient vector of $u$, $\det D^2u$ denotes the determinant of the Hessian matrix of $u$ and $f$ is a known function. These parabolic Monge--Amp\`ere equations also occur in the studies of Gauss curvature flows \cite{Chow-Tsai-GaussCurvFlow,Schnurer-Smoczyk-GaussCurvFlow}, Gauss--Kronecker curvature \cite{Tso-ParaboMA-Deforming}, stochastic theory \cite{Karatzas-ParboMA-StochasticThy}, etc.

The theorem by J\"orgens \cite{Jorgens} ($n=2$), Calabi \cite{Calabi} ($n\leq 5$)
and Pogorelov \cite{Pogorelov} ($n\geq 2$) states that any classical convex solution to the elliptic Monge--Amp\`ere equation
\[
\det D^2u=1\quad\text{in }\mathbb R^n
\]
must be a quadratic polynomial. Especially when $n=2$, the convexity assumption is not necessary since a classical solution to the Monge--Amp\`ere equation is either convex or concave.   For different proofs and extensions, we refer to Cheng--Yau \cite{ChengandYau}, Caffarelli \cite{Caffarelli-InteriorEstimates-MA},
Jost--Xin \cite{JostandXin}, Fu \cite{Fu-Bernstein-SPL}, Li--Xu--Simon--Jia \cite{Book-Li-Xu-Simon-Jia-MA}, Warren \cite{Warren-Calibrations-MA}, Jia--Li \cite{Jia-Li-AsympMA-halfspace}, etc.
In 1998, Guti\'errez--Huang \cite{Gutierrez-Huang-JCP-ParaboMA} first proved the following Liouville type theorem on the parabolic Monge--Amp\`ere equation
\begin{equation}\label{equ-ParboMA-Krylov}
  -u_t\cdot\det D^2u=1\quad\text{in }\mathbb R^{n+1}_-:=\mathbb R^n\times (-\infty,0],
\end{equation}
extending the J\"orgens--Calabi--Pogorelov theorem on the elliptic Monge--Amp\`ere equation. A function $u:\mathbb R^{n+1}_-\rightarrow\mathbb R$ is called parabolically convex if it is continuous, convex in $x$ and non-increasing in $t$.
\begin{theorem}
  Let $u\in C^{4,2}(\mathbb R^{n+1}_-)$ be a parabolically convex solution to \eqref{equ-ParboMA-Krylov} such that there exist positive constants $m_1$ and $m_2$ with
\begin{equation}\label{equ-cond-BoundDeriva}
  -m_1\leq u_t(x,t)\leq -m_2,\quad\forall~(x,t)\in\mathbb R^{n+1}_-.
\end{equation}
Then up to a re-scaling and transformation,
\[
u(x,t)=-t+\frac{|x|^2}{2}.
\]
\end{theorem}
Later in 2018, Zhang--Bao \cite{Zhang-Bao-Calabi-paraboMA-periodic} further generalized the result to parabolic Monge--Amp\`ere equations with periodic right hand side and reduced regularity assumption to $u\in C^{2,1}(\mathbb R^{n+1}_-)$. For other generalizations of the Liouville theorems on the parabolic Monge--Amp\`ere equation, we refer to  \cite{Loftin-Tsui-AncientAffinNormalFlow,Xiong-Bao-JCP-ParaboMA,
Wang-Bao-Asymparabolic-MA,Zhang-Bao-Wang-JCP-ParabolicMA}, etc.

Historically, most studies focus on the case when the partial derivative $u_t$ satisfies condition \eqref{equ-cond-BoundDeriva}. In the pioneering work \cite{Gutierrez-Huang-JCP-ParaboMA}, they provide a counterexample
\[
u_0(x,t)=\left(
\dfrac{(n+1)^{n+2}}{(2n)^{n}(n-1)}
\right)^{\frac{1}{n+1}}
(-t)^{\frac{1}{n+1}}|x|^{\frac{2n}{n+1}},
\]
which satisfies \eqref{equ-ParboMA-Krylov} in viscosity sense but not in classical sense because  the Hessian matrix $D^2u_0$ is singular at  $x=0$.

Our first main theorem focuses on the existence  of smooth solution $u_1$ to \eqref{equ-ParboMA-Krylov}, which is parabolically convex but dissatisfies condition \eqref{equ-cond-BoundDeriva}.
\begin{theorem}\label{thm-main-0}
  Let $n\geq 2$. There exists a smooth, parabolically convex solution $u_1$ to \eqref{equ-ParboMA-Krylov} dissatisfying \eqref{equ-cond-BoundDeriva}. Especially, $u_1$ is radially symmetric in $x$-variable and is of variable separated form
\begin{equation}\label{equ-SepVar-Sol}
u_1(x,t)=w(t)\cdot\varphi(|x|),\quad\forall~(x,t)\in \mathbb R^{n+1}_-.
\end{equation}
\end{theorem}

Our second main theorem focuses on the asymptotic behavior and classification of such radially symmetric and variable separated   solutions.

\begin{theorem}\label{thm-main-1}
  Let $n\geq 2$ and $u$
  be a smooth, parabolically convex  solution to \eqref{equ-ParboMA-Krylov} of variable separated form \eqref{equ-SepVar-Sol}. Then up to a re-scaling and transformation, $u=u_1$ and we have
  \[
  w(t)=(n+1)^{\frac{1}{n+1}}(-t)^{\frac{1}{n+1}}\quad\text{and}\quad
  \lim_{r\rightarrow\infty}\dfrac{\varphi(r)}{r^{\frac{2n}{n+1}}}
  =\dfrac{n+1}{(2n)^{\frac{n}{n+1}}(n-1)^{\frac{1}{n+1}}},
  \]
  i.e., for any $t<0$,
  \[
  \dfrac{u(x,t)}{u_0(x,t)}\rightarrow 1,\quad\text{as }|x|\rightarrow\infty.
  \]
  The derivatives of $\varphi$ share similar asymptotic behavior as below
  \[
  \lim_{r\rightarrow\infty}\dfrac{\varphi'(r)}{\frac{2n}{n+1}r^{\frac{n-1}{n+1}}}=
  \lim_{r\rightarrow\infty}\dfrac{\varphi''(r)}{\frac{2n(n-1)}{(n+1)^2}r^{\frac{-2}{n+1}}}=
  \dfrac{n+1}{(2n)^{\frac{n}{n+1}}(n-1)^{\frac{1}{n+1}}}.
  \]
\end{theorem}

Theorems \ref{thm-main-0} and \ref{thm-main-1} are existence theorem and Liouville type theorem of entire solutions respectively.
We would like to mention that the results on entire solutions may be very different from solutions on exterior domain and on bounded domains. See for instance  Dai \cite{Dai-ParboMA-ExterDP-1},  Wang--Bao \cite{Wang-Bao-ExtireKHessian}, Liu--Wang--Bao \cite{Liu-Wang-Bao-ExisEntireSol-LagMeanCur}, etc.

Furthermore, when $n\geq 4$, the asymptotic behavior of $\varphi$ can be further refined.

\begin{theorem}\label{thm-main-refined}
  Let $n\geq 4$ and $u_1$ be the smooth,  parabolically convex  solution to \eqref{equ-ParboMA-Krylov} constructed in Theorem \ref{thm-main-0}. Then there exists a dimensional constant $K_n>0$ such that
  \[
  \varphi(r)=\dfrac{n+1}{(2n)^{\frac{n}{n+1}}(n-1)^{\frac{1}{n+1}}}r^{\frac{2n}{n+1}}
  +O(r^{-K_n})\quad\text{as }r\rightarrow\infty.
  \]
  Furthermore, $K_n$ can be given explicitly as
  \begin{equation}\label{equ-Def-Kn}
  K_n=
  \left\{
  \begin{array}{llll}
  \frac{n(n-3)}{2(n+1)}, & \text{if }n=4,5,\\
  \frac{n}{n+1}\left(\frac{n-3}{2}-\sqrt{\frac{(n-3)^2}{4}-\frac{2(n-1)}{n}}\right), & \text{if }n\geq 6.\\
  \end{array}
  \right.
  \end{equation}
\end{theorem}

\begin{remark}\label{Rem-Dim1Case}
  When $n=1$, such a smooth, parabolically convex solution dissatisfying \eqref{equ-cond-BoundDeriva} still exists and is unique up to a re-scaling and transformation. However, the asymptotic behavior of $\varphi$ becomes
  \[
  \lim_{r\rightarrow\infty}\dfrac{\varphi^2(r)}{r^2\ln\varphi(r)}=2.
  \]
  When $n=2$ or $3$, it remains unknown whether the asymptotic behavior proved in Theorem \ref{thm-main-1} can be further refined into a similar form as in Theorem \ref{thm-main-refined}.
\end{remark}

The paper is organized as follows. In section \ref{sec-Existence}, we construct a smooth parabolically convex entire solution that dissatisfies \eqref{equ-cond-BoundDeriva} by constructing Euler's break line and prove Theorem \ref{thm-main-0}. In section \ref{sec-AsymAnal}, we analyze the asymptotics at infinity and prove Theorem \ref{thm-main-1}. In section \ref{sec-RefinedAsym}, we provide higher order asymptotic behavior of the solution $u_1$  and prove Theorem \ref{thm-main-refined} by introducing an intermediate variable and linearizing the equation. In section \ref{sec-Dim1}, we prove Remark \ref{Rem-Dim1Case}.

\section{Existence and uniqueness  of  entire solution with unbounded $u_t$}\label{sec-Existence}

Hereinafter, we let $n\geq 2$ unless stated otherwise.
Let $r:=|x|$ and $u$ be a variable separated  function in the form of \eqref{equ-SepVar-Sol}. By a direct computation, for all $r>0$ and $t\leq 0$,
\[
\det D^2u(x,t)=w^n(t)\cdot \varphi''(r)\cdot\left(\dfrac{\varphi'(r)}{r}\right)^{n-1}\quad\text{and}
\quad -u_t(x,t)=-w'(t)\cdot\varphi(r).
\]
Thus a necessary condition such that $u$ satisfies equation \eqref{equ-ParboMA-Krylov} is
\begin{equation}\label{equ-SepVar-Transformed}
-w^{\prime}(t) w^n(t) \cdot \varphi(r) \varphi^{\prime \prime}(r)\left(\frac{\varphi^{\prime}(r)}{r}\right)^{n-1}=1,\quad\forall~r>0,~t<0.
\end{equation}
Since \eqref{equ-SepVar-Transformed} is variable separated, there exists a positive constant $C>0$ such that
\begin{equation}\label{equ-SepVar-TwoEqu}
  \left\{
  \begin{array}{lllll}
    -w^{\prime}(t) w^n(t)=C, & t<0,\\
    \displaystyle \varphi(r) \varphi^{\prime \prime}(r)\left(\frac{\varphi^{\prime}(r)}{r}\right)^{n-1}=C^{-1}, & r>0.\\
  \end{array}
  \right.
\end{equation}
Without loss of generality, we may consider only the case  $C=1$, otherwise we re-scale by
\[
\tilde w(t):=C^{-\frac{1}{n+1}}w(t)\quad\text{and}\quad
\tilde\varphi(r):=C^{\frac{1}{n+1}}\varphi(r),
\]
in which case, $\tilde w$ and $\tilde \varphi$ satisfies system \eqref{equ-SepVar-TwoEqu} with the desired $C$ and
\[
\tilde u(x,t):=\tilde w(t)\cdot\tilde\varphi(|x|)=w(t)\cdot\varphi(|x|)=u(x,t).
\]

The first equation in \eqref{equ-SepVar-TwoEqu} is solved explicitly by
\[
w(t)=\left(w^{n+1}(0)-(n+1)t\right)^{\frac{1}{n+1}},\quad\text{where}\quad w(0)\geq 0.
\]
The second equation in \eqref{equ-SepVar-TwoEqu} can be transformed into
\[
\varphi\cdot (\varphi')^{n-1}\cdot\varphi''=r^{n-1},\quad\text{i.e.,}\quad
((\varphi')^n)'=\dfrac{nr^{n-1}}{\varphi},\quad\forall~r>0.
\]
Furthermore, to guarantee the differentiability of $u$ near the origin $x=0$ and the convexity, we have
\[
\varphi'(0)=0,\quad \varphi'(r)\geq 0\quad\text{and}\quad\varphi''(r)\geq 0,\quad\forall~r\geq 0.
\]
Up to a transformation in $t$-variable and the value of $u$, we consider only the case
\begin{equation}\label{equ-IniVal}
w(0)=0\quad\text{and}\quad \varphi(0)=1.
\end{equation}
Thus hereinafter, we focus on the following initial value problem of  ordinary differential equation
\begin{equation}\label{equ-Sys-InteForm}
  \left\{
  \begin{array}{llll}
    \displaystyle (\varphi'(r))^n=n\int_0^r\dfrac{s^{n-1}}{\varphi(s)}\mathrm ds, & r>0,\\
    \varphi(0)=1.\\
  \end{array}
  \right.
\end{equation}
The solution given by  initial value problem \eqref{equ-IniVal} is named as $u_1$.
We shall further prove that the solution to initial value problem \eqref{equ-Sys-InteForm} is also unique. Consequently, any  parabolically convex,  radially symmetric, variable separated solution  to equation \eqref{equ-ParboMA-Krylov} must be $u_1$ under suitable  re-scaling and transformation.

\subsection{Local existence result}

In this subsection, we prove that initial value problem \eqref{equ-Sys-InteForm} can be solved locally near $r=0$. Our main strategy is to construct Euler's break line and obtain a $\epsilon$-approximation solution.

\begin{lemma}\label{Lem-LocalExis}
  There exists a  solution $\varphi\in C^1([0,1])$ satisfying initial value problem \eqref{equ-Sys-InteForm} on $[0,1]$.
\end{lemma}
\begin{proof}
  For reading simplicity, we denote
  \[
  F(r,\varphi):=n^{\frac{1}{n}}\left(\int_0^r\dfrac{s^{n-1}}{\varphi(s)}\mathrm ds\right)^{\frac{1}{n}}.
  \]
  By a direct computation,
  \[
  F(r,\varphi)\leq r,\quad\forall~r>0,~\varphi\geq 1.
  \]
  For any large $m\in\mathbb N^+$ to be determined, we set
  \[
  r_0:=0\quad\text{and}\quad r_i:=r_{i-1}+\dfrac{1}{m}=\dfrac{i}{m},\quad\forall~i=1,2,\cdots,m
  \]
  and construct Euler break line $\psi$ with initial value $\psi(0):=1$,
  \[
  \psi(r):=\psi(r_{i-1})+F(r_{i-1},\psi)\cdot (r-r_{i-1}),\quad\forall~ r_{i-1}\leq r< r_i,\quad i=1,2,\cdots,m.
  \]
  Notice that $F(r,\varphi)$ relies only on the value of $\varphi$ on $(0,r)$, thus $\psi$ is well-defined.

 For any $r_{i-1}< r<r_i$, $i=1,2,\cdots,m$,
  \[
  \left|\psi'(r)-F(r,\psi)\right|=
    n^{\frac{1}{n}}\left(\left(\int_0^{r}
    \dfrac{s^{n-1}}{\psi(s)}\mathrm ds\right)^{\frac{1}{n}}-
    \left(\int_0^{r_{i-1}}
    \dfrac{s^{n-1}}{\psi(s)}\mathrm ds\right)^{\frac{1}{n}}\right).
  \]
  By the convex property,
  \[
  \left\{
  \begin{array}{lllll}
  (a-b)^n+b^n\leq a^n, & \forall~a>b>0,\\
  (b-a)^n+a^n\leq b^n, & \forall~b\geq a>0,\\
  \end{array}
  \right.
  \]
  i.e.,
  \begin{equation}\label{equ-ConvProp}
  |a-b|\leq  |a^n-b^n|^{\frac{1}{n}},\quad\forall~a, b>0.
  \end{equation}
  Consequently, we continue to have
  \[
  \left|\psi'(r)-F(r,\psi)\right|\leq n^{\frac{1}{n}}\left(\int_{r_{i-1}}^r\dfrac{s^{n-1}}{\psi(s)}\mathrm ds\right)^{\frac{1}{n}}.
  \]
  For any $\epsilon>0$, we may choose $m$ sufficiently large such that  for any $r_{i-1}\leq r<r_i$, $i=1,2,\cdots,m$,
  \[
   \left|\psi'(r)-F(r,\psi)\right|\leq (r^n-r_{i-1}^n)^{\frac{1}{n}}\leq
   \dfrac{2^{\frac{1}{n}}}{m^{\frac{1}{n}}}<\epsilon,
  \]
  where we used the fact that $\psi\geq 1$ for all $r>0.$

  Consequently, for any positive sequence $\{\epsilon_j\}_{j=1}^{\infty}$ with $\epsilon_j\rightarrow0$ as $j\rightarrow\infty$, there exists a sequence of Euler's break line $\{\psi_j\}_{j=1}^{\infty}$ with
  \[
  \left|\psi_j'(r)-F(r,\psi_j)\right|<\epsilon_j,\quad a.e.\quad\text{in } (0,1).
  \]
  Especially since $F(r,\varphi)\leq 1$ for all $r\in [0,1]$ and $\varphi\geq 1$, $\{\psi_j\}_{j=1}^{\infty}$ are uniformly bounded, equicontinous and satisfies  integral inequality
  \[
 \left| \psi_j(r)-\psi_j(0)-n^{\frac{1}{n}}\int_0^r\left(
  \int_0^s\dfrac{t^{n-1}}{\psi_j(t)}\mathrm dt
  \right)^{\frac{1}{n}}\mathrm ds\right|\leq\epsilon_j,\quad\forall~r\in (0,1).
  \]
  By the Arzela--Ascoli theorem, there exist $\psi_0\in C^0([0,1])$ and a subsequence (still denoted as $\{\psi_j\}$) such that
  \[
  \psi_j\rightarrow\psi_0\quad\text{in }C^0([0,1]),\quad\text{as }j\rightarrow\infty.
  \]
  Especially, $\psi_0(0)=1$, $\psi_0\geq 1$ in $[0,1]$ and we shall prove that
  \[
  \left|\int_0^s\left(
  \int_0^t\dfrac{t^{n-1}}{\psi_j(t)}\mathrm dt
  \right)^{\frac{1}{n}}\mathrm ds-
  \int_0^s\left(
  \int_0^t\dfrac{t^{n-1}}{\psi_0(t)}\mathrm dt
  \right)^{\frac{1}{n}}\mathrm ds
  \right|\rightarrow 0\quad\text{as }j\rightarrow\infty.
  \]
  In fact, by the convex property \eqref{equ-ConvProp} again, as $j\rightarrow \infty,$
  \[
  \begin{array}{lll}
  &\displaystyle \left|\int_0^r\left(
  \int_0^s\dfrac{t^{n-1}}{\psi_j(t)}\mathrm dt
  \right)^{\frac{1}{n}}\mathrm ds-
  \int_0^r\left(
  \int_0^s\dfrac{t^{n-1}}{\psi_0(t)}\mathrm dt
  \right)^{\frac{1}{n}}\mathrm ds
  \right| \\
   \leq &\displaystyle \int_0^r\left(
  \int_0^st^{n-1}\dfrac{\left|\psi_j(t)-\psi_0(t)\right|}{\psi_j(t)\psi_0(t)}\mathrm dt
  \right)^{\frac{1}{n}}\mathrm ds\\
  \leq & \displaystyle
   \int_0^r\left(
  \int_0^st^{n-1}\mathrm dt
  \right)^{\frac{1}{n}}\mathrm ds\cdot \max_{t\in[0,1]}\left|\psi_j(t)-\psi_0(t)\right|
  \rightarrow 0.\\
  \end{array}
  \]
  Consequently, $\psi_0$ is a continuous, monotone non-decreasing solution to the integral equation
  \[
  \psi_0(r)=1+n^{\frac{1}{n}}\int_0^r\left(
  \int_0^s\dfrac{t^{n-1}}{\psi_0(t)}\mathrm dt
  \right)^{\frac{1}{n}}\mathrm ds,\quad\forall~r\in (0,1).
  \]
  By the Newton--Leibnitz formula, this is equivalent to the initial value problem \eqref{equ-Sys-InteForm}.
\end{proof}

\subsection{Proof of theorem \ref{thm-main-0}}

In this subsection, we prove that the solution found in Lemma \ref{Lem-LocalExis} can be extended to   $(0,+\infty)$. According to the proof above, we only need to prove that  $\varphi$ is locally bounded on $(0,+\infty)$.

\begin{lemma}\label{Lem-GlobalExis}
  For any $R>0$, suppose $\varphi$ is a monotone non-decreasing solution to the initial value problem \eqref{equ-Sys-InteForm}, then
  \[
  1\leq\varphi(r)\leq 1+\dfrac{1}{2}R^{2},\quad\forall~r\in [0,R].
  \]
\end{lemma}
\begin{proof}
  By the equation in \eqref{equ-Sys-InteForm}, we have
  \[
  (\varphi'(r))^n\leq n\int_0^rs^{n-1}\mathrm ds\leq r^n\quad\text{i.e.,}\quad
  \varphi'(r)\leq r,\quad\forall~r\in (0,R).
  \]
  Together with the initial value condition $\varphi(0)=1$, we have
  \[
  \varphi(r)\leq \frac{1}{2}r^2+1,\quad\forall~r\in [0,R].
  \]
  This finishes the proof of the desired local bound immediately.
\end{proof}

As a direct consequence of Lemma \ref{Lem-GlobalExis}, $\varphi\in C^1([0,\infty))$ is a classical monotone non-decreasing solution to \eqref{equ-Sys-InteForm}.

\begin{proof}[Proof of Theorem \ref{thm-main-0}]
  By Lemmas \ref{Lem-LocalExis} and \ref{Lem-GlobalExis}, there exists $\varphi\in C^1([0,+\infty))$ satisfying  \eqref{equ-Sys-InteForm}. By the L'Hospital's rule,
  \[
  \lim_{r\rightarrow 0}\dfrac{(\varphi'(r))^n}{r^n}=
  \lim_{r\rightarrow 0}\dfrac{1}{r^n}n\int_0^r\dfrac{s^{n-1}}{\varphi(s)}\mathrm ds=\lim_{r\rightarrow 0}\dfrac{1}{\varphi(r)}=1.
  \]
  Thus by taking partial derivative to \eqref{equ-Sys-InteForm}, dividing by $r^{n-1}$ and sending $r\rightarrow\infty$, we have
  \[
  \lim_{r\rightarrow 0}\dfrac{n(\varphi'(r))^{n-1}}{r^{n-1}}\varphi''(r)=\lim_{r\rightarrow 0}n\dfrac{1}{\varphi(r)}=n,\quad\text{i.e.,}\quad \lim_{r\rightarrow 0}\varphi''(r)=1.
  \]
  Hence $\varphi\in C^2([0,\infty))$ and satisfies  the initial value problem
  \begin{equation}\label{equ-Sys-origin}
  \left\{
  \begin{array}{llll}
    \varphi\cdot (\varphi')^{n-1}\cdot\varphi''=r^{n-1}, & r>0,\\
    \varphi'(0)=0,\\
    \varphi(0)=1.
  \end{array}
  \right.
  \end{equation}
  Thus
  \[
  u_1(x,t):=(-(n+1)t)^{\frac{1}{n+1}}\cdot\varphi(|x|)
  \]
  is a classical parabolically convex solution to \eqref{equ-ParboMA-Krylov},  $u_1\in C^2(\mathbb R^{n+1}_-)$ and dissatisfies  \eqref{equ-cond-BoundDeriva}.

  It remains to prove that $u_1$ is smooth near $x=0$, then the smoothness of $u_1$ follows similarly from the proof. Notice that $u_1$ is variable separated with
  \[
  -(u_1)_t=-w'(t)\varphi(r)=(n+1)^{\frac{1}{n+1}-1}(-t)^{-\frac{n}{n+1}}\varphi(r)
  =\dfrac{1}{n+1}(-t)^{-1}u_1.
  \]
  Thus for any given $t<0$, $u_1(\cdot,t)$ is a classical solution to
  \[
  \det D^2u_1=(n+1)(-t)\dfrac{1}{u_1}\quad\text{in }\mathbb R^n.
  \]
  Since
  \[
  \lim_{r\rightarrow 0}\dfrac{\varphi'(r)}{r}=
  \lim_{r\rightarrow 0}\varphi''(r)=1,
  \]
  the Hessian matrix of $D^2u_1$ is positive and bounded in a neighbourhood of origin $x=0$. Since $u_1$ is parabolically convex, $\varphi'(r)$ is strictly positive and hence $D^2u_1$ is locally positive and bounded in any dense subset in $\mathbb R^n.$
  By a direct computation,
  \[
  D\frac{1}{u_1}=-\dfrac{Du_1}{u_1},\quad D^2\dfrac{1}{u_1}=-\dfrac{u_1\cdot D^2u_1+Du_1\times Du_1}{u_1^2},\quad\cdots.
  \]
  Since $u_1(\cdot,t)\in C^2(B_1)$ and $u_1\geq 1$, we have by induction that  for all $k\in\mathbb N$,  $u_1(\cdot,t)\in C^k(B_1)$ implies that $\frac{1}{u_1}(\cdot,t)\in C^k(B_1)$.
  The smoothness of $u_1$ follows from  the regularity theory of the Monge--Amp\`ere equation, see for instance  Caffarelli \cite{Caffarelli-InteriorEstimates-MA}, Jian--Wang \cite{Jian-Wang-ContinuityEstimate-MA}, Figalli--Jhaveri--Mooney \cite{Figalli-Jhaveri-Mooney-Replacement-JianWang}, etc.
\end{proof}

\subsection{Uniqueness of solution}
We only need to prove that for any $R>0$, the solution is unique on $[0,R]$, then the desired uniqueness on $[0,\infty)$ follows immediately. Recall that the initial value  problem \eqref{equ-Sys-InteForm}  is equivalent to the following  neutral functional differential equation (NFDE) with unbounded delay \cite{Book-Lakshm-Wen-Zhang-ThyofODE-UnbounDelay}
\begin{equation}\label{equ-temp4}
\left\{
\begin{array}{llll}
  \varphi'(r)=F(r,\varphi), & r>0,\\
  \varphi(0)=1.
\end{array}
\right.
\end{equation}
Suppose $\varphi$ and $\psi$ are two different monotone non-decreasing solutions to the initial value above, then for any given $r\in [0,R]$, we set
\[
g(k):=n^{\frac{1}{n}}\left(\int_0^r\dfrac{s^{n-1}}{\varphi(s)+k(\psi(s)-\varphi(s))}
\mathrm ds\right)^{\frac{1}{n}},\quad k\in [0,1].
\]
Then $g(0)=F(r,\varphi)$, $g(1)=F(r,\psi)$ and the  derivative of $g$ is
\[
g'(k)=-n^{\frac{1}{n}-1}\left(
\int_0^r\dfrac{s^{n-1}}{\varphi+k(\psi-\varphi)}\mathrm ds
\right)^{\frac{1}{n}-1}\cdot \int_0^r\dfrac{s^{n-1}}{(\varphi+k(\psi-\varphi))^2}\cdot(\psi-\varphi)\mathrm ds.
\]
Consequently, applying the Newton--Leibnitz formula to $g$, there exist constants $C>0$ independent of $r\in[0,R]$ (which may vary from line to line)  such that
\begin{equation}\label{equ-temp5}
\begin{array}{lllll}
&|F(r,\varphi)-F(r,\psi)|=|g(1)-g(0)|\\
\leq &\displaystyle
n^{\frac{1}{n}-1}\int_0^1 \left(
\int_0^r\dfrac{s^{n-1}}{\varphi+k(\psi-\varphi)}\mathrm ds
\right)^{\frac{1}{n}-1}\cdot \int_0^r\dfrac{s^{n-1}}{(\varphi+k(\psi-\varphi))^2}\cdot |\psi-\varphi|\mathrm ds\mathrm dk\\
\leq & \displaystyle C\int_0^1 \left(\int_0^r \dfrac{s^{n-1}}{1+\frac{1}{2}R^2}\mathrm ds\right)^{\frac{1}{n}-1}
\cdot \int_0^r s^{n-1}\mathrm ds\mathrm dk\cdot ||\varphi-\psi||_{C^0([0,r])}\\
\leq &\displaystyle C\int_0^1r^{1-n}\cdot r^n\mathrm dk \cdot ||\varphi-\psi||_{C^0([0,r])}\\
\leq & Cr||\varphi-\psi||_{C^0([0,r])},\\
\end{array}
\end{equation}
where we used the monotonicity and the locally uniform bound as in Lemma \ref{Lem-GlobalExis} i.e.,
\[
1\leq \varphi(r),~\psi(r)\leq 1+\frac{1}{2}R^2,\quad\forall~r\in [0,R].
\]
By the uniqueness result of initial value problem of functional differential equation, such as Theorem 3.2.2 in \cite{Book-Lakshm-Wen-Zhang-ThyofODE-UnbounDelay} and  Theorem 2.1 in \cite{Book-Coddington-Levinson-ODE}, we have $\varphi=\psi$ on $[0,R]$.
For reading simplicity, we  include a proof here.
\begin{proof}[Proof of the uniqueness of solution]
  Integrating both sides of \eqref{equ-temp4}, we find that $\varphi$ and $\psi$ satisfy
  \[
  \varphi(r)=1+\int_0^rF(s,\varphi)\mathrm ds,\quad\psi(r)=1+\int_0^rF(s,\psi)\mathrm ds.
  \]
  Consequently, by the triangle inequality,
  \[
  |(\varphi(r)-\psi(r))-(\varphi(0)-\psi(0))|\leq \int_0^r|F(s,\varphi)-F(s,\psi)|\mathrm ds.
  \]
  Let
  \[
  k(r):=|\varphi(r)-\psi(r)|
  \quad\text{and}\quad \tilde k(r):=\max_{0\leq s\leq r}|\varphi(s)-\psi(s)|,
  \]
  together with estimate \eqref{equ-temp5}, we have
  \[
  k(r)\leq C\int_0^rs||\varphi-\psi||_{C^0([0,s])}\mathrm ds=
  C\int_0^rs\tilde k(s)\mathrm ds.
  \]
  Since the right hand side is monotone increasing with respect to $r$, we have
  \[
  \tilde k(r)=\max_{0\leq s\leq r} k(s)\leq C
  \max_{0\leq s\leq r}  \int_0^st\tilde k(t)\mathrm dt
  =C\int_0^rs\tilde k(s)\mathrm ds,\quad\forall~r\in [0,R].
  \]
  Let
  \[
  K(r):=\int_0^rs\tilde k(s)\mathrm ds,
  \]
  then  $K(0)=0$, $K$ is monotone increasing   and satisfies
  \[
  K'(r)=r\tilde k(r)\leq CrK(r)\leq CRK(r),\quad\forall~r\in (0,R).
  \]
  Consequently, multiplying both sides by $e^{-CRr}$ and integrating on $(0,r)$, we have
  \[
  (e^{-CRr}K(r))'\leq 0,\quad\forall~r\in (0,R)
  \]
  and hence
  \[
  0\leq e^{-CRr}K(r)\leq e^{0}K(0)=0,\quad\text{i.e.,}\quad K(r)\equiv 0,\quad\forall~r\in [0,R].
  \]
  Since $k$ and $\tilde k$ are also non-negative, we have
  \[
  k(r)=\tilde k(r)=0,\quad\forall~r\in [0,R],
  \]
  which implies that $\varphi$ and $\psi$ are the same  on $[0,R]$ and finishes the proof.
\end{proof}

\section{Asymptotics of radially symmetric variable separated solutions}\label{sec-AsymAnal}

In this section, we reveal the asymptotic behavior of $u_1$  at infinity and prove Theorem \ref{thm-main-1}.
Based on the initial value problem  \eqref{equ-Sys-InteForm}, we prove that a lower bound of $\varphi$ provides an upper bound of $\varphi$ and vice versa. Repeatedly using these two arguments, we iterate and prove upper and lower estimates with the desired power of $r$.

\begin{lemma}\label{Lem-Low2High-attempt}
  Let $\varphi$ be the positive non-decreasing solution to \eqref{equ-Sys-InteForm}. Suppose there exist $C_1>0$ and $0\leq k\leq \frac{2n}{n+1}$ such that
  \begin{equation}\label{equ-lowerBD}
  \varphi(r)\geq C_1r^k,\quad\forall~r\geq 0.
  \end{equation}
  Then
  \begin{equation}\label{equ-temp2}
  \varphi(r)\leq 1+C_1^{-\frac{1}{n}}\left(\dfrac{n+1}{n-1}\right)^{\frac{1}{n}}\cdot \dfrac{n+1}{2n}r^{2-\frac{k}{n}},\quad\forall~r\geq 0.
  \end{equation}
\end{lemma}
\begin{proof}
The lower bound of $\varphi$ implies
  \[
    (\varphi'(r))^n = n\int_0^r\dfrac{s^{n-1}}{\varphi(s)}\mathrm ds\leq  C_1^{-1}\dfrac{n}{n-k}r^{n-k},\quad\forall~r>0.
  \]
  Together with the initial value condition $\varphi(0)=1$,
\[
\varphi(r)\leq 1+C_1^{-\frac{1}{n}}\left(\dfrac{n}{n-k}\right)^{\frac{1}{n}}\cdot
\dfrac{n}{2n-k}r^{\frac{2n-k}{n}}.
\]
This finishes the proof of this lemma since $k\leq \frac{2n}{n+1}$.
\end{proof}

\begin{lemma}\label{Lem-High2Low-attempt}
  Let $\varphi$ be the positive non-decreasing solution to \eqref{equ-Sys-InteForm}.
  Suppose there exist $C_2>0$ and $\frac{2n}{n+1}\leq k\leq n$ such that
  \begin{equation}\label{equ-upperBD}
  \varphi(r)\leq 1+C_2r^k,\quad\forall~r\geq 0.
  \end{equation}
  Then
  \[
  \varphi(r)\geq\dfrac{1}{2(1+C_2)^{\frac{1}{n}}} r^{2-\frac{k}{n}},\quad\forall~r\geq 0.
  \]
\end{lemma}
\begin{proof}
 The upper bound of $\varphi$ implies
 \[
   (\varphi'(r))^n \geq  n\int_0^r\dfrac{s^{n-1}}{1+C_2s^k}\mathrm ds
 = \dfrac{r^n}{1+C_2r^k}+C_2k
   \int_0^r\dfrac{s^{n+k-1}}{(1+C_2s^k)^2}\mathrm ds.
 \]
 Since $C_2>0$, it follows immediately that
 \[
 \varphi'(r)\geq \dfrac{r}{(1+C_2r^k)^{\frac{1}{n}}},\quad\forall~r>0.
 \]
 Together with the initial value condition $\varphi(0)=1$, we have
 \[
   \varphi(r) \geq  1+\int_0^r\dfrac{s}{(1+C_2s^k)^{\frac{1}{n}}}\mathrm ds=
    1+\dfrac{1}{2}\dfrac{r^2}{(1+C_2r^k)^{\frac{1}{n}}}
   +\dfrac{C_2k}{2n}\int_0^r\dfrac{s^{2+k}}{(1+C_2s^k)^{\frac{n+1}{n}}}\mathrm ds.
 \]
 As a consequence,
 \[
 \varphi(r)\geq 1+\dfrac{1}{2}\dfrac{r^2}{(1+C_2r^k)^{\frac{1}{n}}},\quad\forall~r\geq 0.
 \]
 Separately discuss $r\leq 1$ and $r>1$ two cases, we find
 \[
 \varphi(r)\geq \left\{
 \begin{array}{lll}
   1, & 0\leq r\leq 1,\\
   \dfrac{1}{2(1+C_2)^{\frac{1}{n}}} r^{2-\frac{k}{n}}, & r>1,\\
 \end{array}
 \right.\geq \dfrac{1}{2(1+C_2)^{\frac{1}{n}}} r^{2-\frac{k}{n}},\quad\forall~r\geq 0.
 \]
 This finishes the proof of this lemma.
\end{proof}

Combining these two lemmas, we obtain an iteration scheme as below.
\begin{lemma}\label{Lem-iterSchem}
  Let $\varphi$ be the monotone non-decreasing solution to \eqref{equ-Sys-InteForm}. Suppose there exist $0<C_1\leq 1$ and $0\leq k\leq \frac{2n}{n+1}$ such that the lower bound \eqref{equ-lowerBD} holds. Then for all $r\geq 0$,
  \[
  \varphi(r)\geq 2^{-\frac{n+1}{n}} \left(\dfrac{n+1}{n-1}\right)^{-\frac{1}{n^2}}\cdot
  \left(\dfrac{n+1}{n}\right)^{-\frac{1}{n}} C_1^{\frac{1}{n^2}} r^{\frac{2n-2}{n}+\frac{k}{n^2}}.
  \]
\end{lemma}
\begin{proof}
  By the result in Lemma \ref{Lem-Low2High-attempt}, $\varphi$ satisfies upper bound estimate \eqref{equ-temp2}. Consequently, $\varphi$ satisfies the condition \eqref{equ-upperBD} in Lemma \ref{Lem-High2Low-attempt} with
  \[
  C_2=C_1^{-\frac{1}{n}}\left(\dfrac{n+1}{n-1}\right)^{\frac{1}{n}}\cdot \dfrac{n+1}{2n}\quad\text{and}\quad k'=\dfrac{2n-k}{n}<2\leq n.
  \]
  By Lemma \ref{Lem-High2Low-attempt}, we have
  \[
  \varphi(r)\geq \dfrac{1}{2}\left(1+C_1^{-\frac{1}{n}}\left(\dfrac{n+1}{n-1}\right)^{\frac{1}{n}}\cdot \dfrac{n+1}{2n}\right)^{-\frac{1}{n}} r^{\frac{2n^2-2n+k}{n^2}},\quad\forall~s\geq 0.
  \]
  Since $0<C_1\leq 1$ and $n\geq 2$, the coefficient satisfies
  \[
  \begin{array}{llll}
  \displaystyle \dfrac{1}{2}\left(1+C_1^{-\frac{1}{n}}\left(\dfrac{n+1}{n-1}\right)^{\frac{1}{n}}\cdot \dfrac{n+1}{2n}\right)^{-\frac{1}{n}}&
  \geq & \dfrac{1}{2}\left(2C_1^{-\frac{1}{n}}\left(\dfrac{n+1}{n-1}\right)^{\frac{1}{n}}\cdot \dfrac{n+1}{n}\right)^{-\frac{1}{n}}\\
  &=& 2^{-\frac{n+1}{n}} \left(\dfrac{n+1}{n-1}\right)^{-\frac{1}{n^2}}\cdot
  \left(\dfrac{n+1}{n}\right)^{-\frac{1}{n}} C_1^{\frac{1}{n^2}}.
  \end{array}
  \]
  This finishes the proof of this lemma.
\end{proof}

Applying the iteration scheme in Lemma \ref{Lem-iterSchem} repeatedly, we prove the following estimate.
\begin{proposition}\label{Prop-BD-separated}
  Let $\varphi$ be the monotone non-decreasing solution to \eqref{equ-Sys-InteForm}. Then there exist positive dimensional constants $C_3$ and $ C_4$ such that
  \begin{equation}\label{equ-PoweEst}
  C_3r^{\frac{2n}{n+1}}\leq \varphi(r)\leq 1+C_4r^{\frac{2n}{n+1}},\quad\forall~r\geq 0.
  \end{equation}
\end{proposition}
\begin{proof}
  Starting from
  \[
  \varphi(r)\geq 1,\quad\forall~r\geq 0,
  \]
  the result in Lemma \ref{Lem-iterSchem} implies
  \[
  \varphi(r)\geq p_1r^{k_1},\quad\forall~r\geq 0,
  \]
  where
  \[
  p_1:=2^{-\frac{n+1}{n}} \left(\dfrac{n+1}{n-1}\right)^{-\frac{1}{n^2}}\cdot
  \left(\dfrac{n+1}{n}\right)^{-\frac{1}{n}} \quad\text{and}\quad
  k_1:=\dfrac{2n-2}{n}.
  \]
  By iteration, we have two constant sequences $\{p_m\}_{m=1}^{\infty}$ and $\{k_m\}_{m=1}^{\infty}$ such that
  \[
  \varphi(r)\geq p_mr^{k_m},\quad\forall~r\geq 0,
  \]
  where
  \[
  p_m:=p_1\cdot p_{m-1}^{\frac{1}{n^2}}\quad\text{and}\quad k_m:=\dfrac{2n-2}{n}+\dfrac{k_{m-1}}{n^2},\quad\forall~m=2,3,\cdots.
  \]
  By a direct computation,
  \[
  p_m=p_1^{1+\frac{1}{n^2}+\cdots+\frac{1}{n^{2(m-1)}}}\rightarrow p_1^{\frac{n^2}{n^2-1}}
  \]
  and
  \[
  k_m=\dfrac{2n-2}{n}\left(1+\dfrac{1}{n^2}+\cdots+\dfrac{1}{n^{2(m-1)}}\right)\rightarrow
  \dfrac{2n}{n+1}
  \]
  as $m\rightarrow\infty.$ This proves the lower bound estimate by sending $m\rightarrow\infty$,
  \[
  \varphi(r)\geq 2^{-\frac{n}{n-1}}\left(\dfrac{n+1}{n-1}\right)^{-\frac{1}{n^2-1}}
  \left(\dfrac{n+1}{n}\right)^{-\frac{n}{n^2-1}}r^{\frac{2n}{n+1}},\quad\forall~r\geq 0.
  \]
  Applying Lemma \ref{Lem-Low2High-attempt}, we similarly have the upper bound estimate of $\varphi$, where the coefficient relies only on dimension $n$.
\end{proof}

Now we are ready to prove the asymptotics of $\varphi$ based on the results in Proposition \ref{Prop-BD-separated}. In the language of $\varphi$, Theorem \ref{thm-main-1} is a direct consequence of the following theorem.
\begin{theorem}\label{thm-main-2}
  Let $\varphi$ be the monotone non-decreasing solution to \eqref{equ-Sys-InteForm}. Then
  \[
  \lim_{r\rightarrow\infty}\dfrac{\varphi(r)}{r^{\frac{2n}{n+1}}}=
    \lim_{r\rightarrow\infty}\dfrac{\varphi'(r)}{\frac{2n}{n+1}r^{\frac{n-1}{n+1}}}=
  \lim_{r\rightarrow\infty}\dfrac{\varphi''(r)}{\frac{2n(n-1)}{(n+1)^2}r^{\frac{-2}{n+1}}}=
  \dfrac{n+1}{(2n)^{\frac{n}{n+1}}(n-1)^{\frac{1}{n+1}}}.
  \]
\end{theorem}
By Theorem \ref{thm-main-2}, there are
two positive constants  $\overline K$ and $\underline K$  such that
\[
\overline K:=\limsup_{r\rightarrow\infty}\dfrac{\varphi(r)}{r^{\frac{2n}{n+1}}}
\geq \liminf_{r\rightarrow\infty}\dfrac{\varphi(r)}{r^{\frac{2n}{n+1}}}=:\underline K.
\]
\begin{lemma}\label{Lem-High2Low-coef}
  Let $\varphi$ be the positive non-decreasing solution to \eqref{equ-Sys-InteForm}. Then
  \[
  \underline K\geq \overline K^{-\frac{1}{n}}\left( \dfrac{n+1}{n-1}\right)^{\frac{1}{n}}\cdot \frac{n+1}{2n}.
  \]
\end{lemma}
\begin{proof}
  For any $\epsilon>0$, from the definition of $\overline K$, there exists $R\gg 1$ such that
  \[
  \dfrac{\varphi(r)}{r^{\frac{2n}{n+1}}}\leq\overline K+\epsilon,\quad\forall~r\geq R.
  \]
  Consequently, there exist dimensional constants $C>0$ independent of $r>R$ such that,
  \[
  \begin{array}{lllll}
    (\varphi'(r))^n &=& \displaystyle n\int_0^r\dfrac{s^{n-1}}{\varphi(s)}\mathrm ds\\
    &\geq & -C+\displaystyle \dfrac{1}{\overline K+\epsilon}n\int_R^rs^{n-1-\frac{2n}{n+1}}\mathrm ds\\
    &\geq & -C+\displaystyle \dfrac{1}{\overline K+\epsilon}\cdot \dfrac{n+1}{n-1}\left(r^{\frac{n^2-n}{n+1}}-R^{\frac{n^2-n}{n+1}}\right)\\
    &\geq &\displaystyle
     \dfrac{1}{\overline K+\epsilon}\cdot \dfrac{n+1}{n-1}r^{\frac{n^2-n}{n+1}}-\dfrac{1}{\overline K+\epsilon}\cdot \dfrac{n+1}{n-1}R^{\frac{n^2-n}{n+1}}-C.\\
  \end{array}
  \]
Furthermore, there exists $R'>R$ such that for all $r>R'$,
\[
\begin{array}{lll}
  (\varphi'(r))^n &\geq & \displaystyle \dfrac{1}{\overline K+2\epsilon}\cdot \dfrac{n+1}{n-1}r^{\frac{n^2-n}{n+1}},\\
  \varphi(r) &\geq &\displaystyle \varphi(R')+ \left(\dfrac{1}{\overline K+2\epsilon}\cdot \dfrac{n+1}{n-1}\right)^{\frac{1}{n}}\cdot \frac{n+1}{2n} \left(r^{\frac{2n}{n+1}}-R'^{\frac{2n}{n+1}}\right)\\
  &=&\displaystyle  \left(\dfrac{1}{\overline K+2\epsilon}\cdot \dfrac{n+1}{n-1}\right)^{\frac{1}{n}}\cdot \frac{n+1}{2n}r^{\frac{2n}{n+1}}
  -\left(\dfrac{1}{\overline K+2\epsilon}\cdot \dfrac{n+1}{n-1}\right)^{\frac{1}{n}}\cdot \frac{n+1}{2n}R'^{\frac{2n}{n+1}}+\varphi(R').
\end{array}
\]
Sending $r\rightarrow\infty$, we have
\[
\underline K=\liminf_{r\rightarrow\infty}\dfrac{\varphi(r)}{r^{\frac{2n}{n+1}}}
\geq \left(\dfrac{1}{\overline K+2\epsilon}\cdot \dfrac{n+1}{n-1}\right)^{\frac{1}{n}}\cdot \frac{n+1}{2n}.
\]
Since $\epsilon>0$ is an arbitrary constant, the desired result follows by sending $\epsilon\rightarrow 0^+.$
\end{proof}

\begin{lemma}\label{Lem-Low2High-coef}
  Let $\varphi$ be the positive non-decreasing solution to \eqref{equ-Sys-InteForm}. Then
  \[
  \overline K\leq \underline K^{-\frac{1}{n}}\left( \dfrac{n+1}{n-1}\right)^{\frac{1}{n}}\cdot \frac{n+1}{2n}.
  \]
\end{lemma}
\begin{proof}
  For any sufficiently small $\epsilon>0$, from the definition of $\underline K$, there exists $R\gg 1$ such that
  \[
  \dfrac{\varphi(r)}{r^{\frac{2n}{n+1}}}\geq\underline K-\epsilon,\quad\forall~r\geq R.
  \]
  Consequently, there exist dimensional constant $C>0$ independent of $r>R$ such that,
  \[
  \begin{array}{llll}
  (\varphi'(r))^n &=& \displaystyle n\int_0^r\dfrac{s^{n-1}}{\varphi(s)}\mathrm ds\\
  &\leq &\displaystyle C+\dfrac{1}{\underline K-\epsilon} n\int_R^rs^{n-1-\frac{2n}{n+1}}\mathrm ds\\
  &\leq &\displaystyle C+\dfrac{1}{\underline K-\epsilon}\cdot \dfrac{n+1}{n-1}\left(r^{\frac{n^2-n}{n+1}}-R^{\frac{n^2-n}{n+1}}\right)\\
  &\leq &\displaystyle
  \dfrac{1}{\underline K-\epsilon}\cdot \dfrac{n+1}{n-1}r^{\frac{n^2-n}{n+1}}
  -\dfrac{1}{\underline K-\epsilon}\cdot \dfrac{n+1}{n-1}R^{\frac{n^2-n}{n+1}}+C.
  \end{array}
  \]
  Furthermore, there exists $R'>R$ such that for all $r>R'$,
  \[
  \begin{array}{lll}
  (\varphi'(r))^n &\leq & \dfrac{1}{\underline K-2\epsilon}\cdot \dfrac{n+1}{n-1}r^{\frac{n^2-n}{n+1}},\\
    \varphi(r)&\leq &\displaystyle \varphi(R')+ \left(\dfrac{1}{\underline K-2\epsilon}\cdot \dfrac{n+1}{n-1} \right)^{\frac{1}{n}}\cdot \dfrac{n+1}{2n} \left(r^{\frac{2n}{n+1}}-R'^{\frac{2n}{n+1}}\right)\\
    &=& \left(\dfrac{1}{\underline K-2\epsilon}\cdot \dfrac{n+1}{n-1} \right)^{\frac{1}{n}}\cdot \dfrac{n+1}{2n}r^{\frac{2n}{n+1}}-
    \left(\dfrac{1}{\underline K-2\epsilon}\cdot \dfrac{n+1}{n-1} \right)^{\frac{1}{n}}\cdot \dfrac{n+1}{2n}R'^{\frac{2n}{n+1}}+\varphi(R').
  \end{array}
  \]
  Sending $r\rightarrow\infty$, we have
  \[
  \overline K=\limsup_{r\rightarrow\infty}\dfrac{\varphi(r)}{r^{\frac{2n}{n+1}}}
\leq \left(\dfrac{1}{\underline K-2\epsilon}\cdot \dfrac{n+1}{n-1}\right)^{\frac{1}{n}}\cdot \frac{n+1}{2n}.
  \]
  Since $\epsilon>0$ is an arbitrary constant, the desired result follows by sending $\epsilon\rightarrow 0^+.$
\end{proof}

\begin{proof}[Proof of Theorem \ref{thm-main-2}]
  Combining Lemmas \ref{Lem-High2Low-coef} and \ref{Lem-Low2High-coef}, we have
  \[
  \begin{array}{llll}
  \overline K&\leq&\displaystyle \underline K^{-\frac{1}{n}}\left( \dfrac{n+1}{n-1}\right)^{\frac{1}{n}}\cdot \frac{n+1}{2n}\\
  &\leq&\displaystyle \left(\overline K^{-\frac{1}{n}}\left( \dfrac{n+1}{n-1}\right)^{\frac{1}{n}}\cdot \frac{n+1}{2n}\right)^{-\frac{1}{n}}\left( \dfrac{n+1}{n-1}\right)^{\frac{1}{n}}\cdot \frac{n+1}{2n}\\
  &=&\displaystyle \overline K^{\frac{1}{n^2}}
  \left( \dfrac{n+1}{n-1}\right)^{\frac{n-1}{n^2}}\cdot \left(\frac{n+1}{2n}\right)^{\frac{n-1}{n}}\\
  \end{array}
  \]
  and
  \[
  \begin{array}{lll}
    \underline K&\geq &\displaystyle \overline K^{-\frac{1}{n}}\left( \dfrac{n+1}{n-1}\right)^{\frac{1}{n}}\cdot \frac{n+1}{2n}\\
    &\geq &\displaystyle
    \left(\underline K^{-\frac{1}{n}}\left( \dfrac{n+1}{n-1}\right)^{\frac{1}{n}}\cdot \frac{n+1}{2n}\right)^{-\frac{1}{n}}\left( \dfrac{n+1}{n-1}\right)^{\frac{1}{n}}\cdot \frac{n+1}{2n}\\
    &=&\displaystyle \underline K^{\frac{1}{n^2}}
  \left( \dfrac{n+1}{n-1}\right)^{\frac{n-1}{n^2}}\cdot \left(\frac{n+1}{2n}\right)^{\frac{n-1}{n}}.\\
  \end{array}
  \]
  By a direct computation,
  \[
  \dfrac{n+1}{(2n)^{\frac{n}{n+1}}(n-1)^{\frac{1}{n+1}}}\leq \underline K\leq\overline K\leq \dfrac{n+1}{(2n)^{\frac{n}{n+1}}(n-1)^{\frac{1}{n+1}}}.
  \]
  This finishes the proof of the asymptotic behavior of $\varphi$ at infinity. By using the integral form \eqref{equ-Sys-InteForm} and the L'Hospital's rule,
  \[
  \lim_{r\rightarrow\infty}\dfrac{(\varphi'(r))^n}{r^{\frac{n^2-n}{n+1}}}=
  \lim_{r\rightarrow\infty}\dfrac{\displaystyle n\int_0^r\dfrac{s^{n-1}}{\varphi(s)}\mathrm ds}{r^{\frac{n^2-n}{n+1}}}=\dfrac{n+1}{n-1}\lim_{r\rightarrow\infty}
  \dfrac{r^{\frac{2n}{n+1}}}{\varphi(r)}=\left(\frac{2n}{n-1}\right)^{\frac{n}{n+1}}.
  \]
  Furthermore, by using the equation in \eqref{equ-Sys-origin}, we obtain the asymptotic behavior of $\varphi''$ as below
  \[
  \lim_{r\rightarrow\infty}\dfrac{\varphi''(r)}{r^{\frac{-2}{n+1}}}
  =\lim_{r\rightarrow\infty}\dfrac{r^{\frac{2n}{n+1}}\cdot r^{\frac{(n-1)^2}{n+1}}}{\varphi(r)\cdot (\varphi'(r))^{n-1}}=\dfrac{(2n)^{\frac{1}{n+1}}(n-1)^{\frac{n}{n+1}}}{n+1}.
  \]
  This finishes the proof of Theorem \ref{thm-main-2}.
\end{proof}

\section{Refined asymptotic behavior at infinity}\label{sec-RefinedAsym}

In this section, we analyze higher order asymptotic behavior than the results in Theorem \ref{thm-main-2}.
By linearizing the equation and using asymptotic stability analysis, it is equivalent to prove the following refined asymptotic behavior of $\varphi$ when $n\geq 4$.
\begin{lemma}\label{Lem-Refine-1}
  Let $n\geq 4$ and $\varphi$ be the monotone non-decreasing solution to \eqref{equ-Sys-InteForm}. Then there exists dimensional constant $K_n>0$ given in \eqref{equ-Def-Kn} such that
  \[
  \varphi(r)=
  \dfrac{n+1}{(2n)^{\frac{n}{n+1}}(n-1)^{\frac{1}{n+1}}}r^{\frac{2n}{n+1}}+O(r^{-K_n})\quad\text{as }r\rightarrow\infty.
  \]
\end{lemma}

In order to simplify the proof, we introduce an intermediate variable $s(r):=r^{\frac{2n}{n+1}}$.
We rewrite variable separated solution of form \eqref{equ-SepVar-Sol} into
\[
\varphi(r)\cdot w(t)=u(x,t)=\Phi(s(r))\cdot w(t),\quad\forall~r\geq 0,~t\leq 0.
\]
By a direct computation,
\[
\begin{array}{lll}
\varphi'(r)&=&\dfrac{2n}{n+1}\Phi'(s)\cdot r^{\frac{n-1}{n+1}}=
\dfrac{2n}{n+1}\Phi'(s)\cdot s^{\frac{n-1}{2n}},\\
\varphi''(r)&=&\left(\dfrac{2n}{n+1}\right)^2\Phi''(s)\cdot r^{ \frac{2n-2}{n+1}}+ \dfrac{2n(n-1)}{(n+1)^2}\Phi'(s)\cdot r^{-\frac{2}{n+1}}\\
&=& \left(\dfrac{2n}{n+1}\right)^2\Phi''(s)\cdot s^{\frac{n-1}{n}}+ \dfrac{2n(n-1)}{(n+1)^2}\Phi'(s)\cdot s^{-\frac{1}{n}},\\
\end{array}
\]
and hence   equation \eqref{equ-Sys-origin}  is equivalent to
\[
\begin{array}{llll}
& r^{n-1}=s^{\frac{n^2-1}{2n}}\\
=&
  \varphi\cdot (\varphi')^{n-1}\cdot\varphi''\\
  =&\displaystyle \Phi\cdot\left( \left(\frac{2n}{n+1}\right)^{n-1}\cdot (\Phi')^{n-1} \cdot s^{\frac{n^2-2n+1}{2n}}\right)
  \cdot \left(\left(\dfrac{2n}{n+1}\right)^2\Phi''(s)\cdot s^{\frac{n-1}{n}}+ \dfrac{2n(n-1)}{(n+1)^2}\Phi'(s)\cdot s^{-\frac{1}{n}}\right)\\
  =& \displaystyle
  \left(\frac{2n}{n+1}\right)^{n+1} s^{\frac{n^2-1}{2n}} \Phi\cdot (\Phi')^{n-1}\cdot\Phi''
  +\frac{(2n)^n(n-1)}{(n+1)^{n+1}}
  s^{\frac{n^2-2n-1}{2n}} \Phi\cdot (\Phi')^n,
\end{array}
\]
i.e.,
\begin{equation}\label{equ-Transformed}
\Phi\cdot (\Phi')^{n-1}\cdot\Phi''+\frac{n-1}{2n} s^{-1}\Phi\cdot (\Phi')^n=\left(\dfrac{n+1}{2n}\right)^{n+1},\quad\forall~s>0.
\end{equation}

By the asymptotic results in Theorem \ref{thm-main-1},
\[
\lim_{s\rightarrow\infty}\dfrac{\Phi(s)}{s}=\lim_{s\rightarrow\infty}\Phi'(s)=
\dfrac{n+1}{(2n)^{\frac{n}{n+1}}(n-1)^{\frac{1}{n+1}}}=:c_n.
\]
By equation \eqref{equ-Transformed}, we further have $\Phi''=o(s^{-1})$ at infinity.
Then Lemma \ref{Lem-Refine-1} follows directly from the following refined asymptotics of $\Phi$.
\begin{lemma}\label{Lem-Refine-2}
  Let $n\geq 4$ and $\Phi$ be the monotone non-decreasing solution to \eqref{equ-Transformed} with initial value $\Phi(0)=1$.
  Then
  \[
  \Phi(s)=c_ns+O(s^{k_n})\quad\text{as }s\rightarrow\infty,\quad\text{where}\quad
  k_n:=
  \left\{
  \begin{array}{lll}
    -\frac{n-3}{4}, & \text{if }n=4,5,\\
    -\frac{n-3}{4}+\sqrt{\frac{(n-3)^2}{16}-\frac{n-1}{2n}}, & \text{if }n\geq 6.
  \end{array}
  \right.
  \]
\end{lemma}
\begin{proof}
Let
\[
\Psi(s):=\Phi(s)-c_ns,\quad\forall~s>0.
\]
Then by \eqref{equ-Transformed} and the asymptotic behavior of $\Phi$ as $s\rightarrow\infty$, we have
\[
\Psi(s)=o(s),\quad\Psi'(s)=o(1)\quad\text{and}\quad\Psi''(s)=\Phi''(s)=o(s^{-1}).
\]
By a direct computation, \eqref{equ-Transformed} implies that $\Psi$ satisfies
\[
(\Psi+c_ns)\cdot (\Psi'+c_n)^{n-1}\cdot \Psi''+\frac{n-1}{2n}(s^{-1}\Psi+c_n)\cdot (\Psi'+c_n)^n=\left(\dfrac{n+1}{2n}\right)^{n+1},\quad\forall~s>0.
\]
By the binomial theorem,
\[
 (\Psi'+c_n)^n= \sum_{k=0}^n\dfrac{n!}{(n-k)!k!}c_n^{n-k}(\Psi')^k,\quad
 (\Psi'+c_n)^{n-1}= \sum_{k=0}^{n-1}\dfrac{(n-1)!}{(n-1-k)!k!}c_n^{n-1-k}(\Psi')^k.
\]
Thus
\[
\begin{array}{lllll}
  \left(\dfrac{n+1}{2n}\right)^{n+1}&=&\displaystyle
  \sum_{k=0}^{n-1}\left(
\left(\Psi\cdot\Psi''+c_ns\Psi''\right)\cdot\frac{(n-1)!}{(n-1-k)!k!}c_n^{n-1-k}\right)(\Psi')^k\\
&&+\displaystyle \sum_{k=0}^n\left(\dfrac{n-1}{2n}(s^{-1}\Psi+c_n)\cdot \dfrac{n!}{(n-k)!k!}c_n^{n-k}
\right)(\Psi')^k,
\end{array}
\]
i.e.,
\begin{equation}\label{equ-temp6}
\left(\dfrac{n+1}{2n}\right)^{n+1}=\sum_{k=0}^nd_{n,k}(\Psi')^k,
\end{equation}
where the coefficients $d_{n,k}$ are
\[
\left\{
\begin{array}{llll}
 \displaystyle \left(\Psi\cdot\Psi''+c_ns\Psi''\right)\cdot\frac{(n-1)!}{(n-1-k)!k!}c_n^{n-1-k}
  +\dfrac{n-1}{2n}(s^{-1}\Psi+c_n)\cdot \dfrac{n!}{(n-k)!k!}c_n^{n-k}, & k\leq n-1,\\
  \displaystyle \dfrac{n-1}{2n}(s^{-1}\Psi+c_n)\cdot \dfrac{n!}{(n-k)!k!}c_n^{n-k}, & k=n.\\
\end{array}
\right.
\]
Especially,
\[
d_{n,0}
= \frac{n-1}{2n}c_n^{n+1}+\frac{n-1}{2n}c_n^ns^{-1}\Psi
+c_n^{n}s\Psi''+\underbrace{c_n^{n-1}(s^{-1}\Psi)\cdot(s\Psi'')}_{\text{higher order term}},
\]
and
\[
d_{n,1}= o(1)+\frac{n-1}{2n}(c_n+o(1))\cdot n c_n^{n-1}\rightarrow \dfrac{n-1}{2}c_n^n\quad\text{as }s\rightarrow\infty.
\]
When $k\geq 2$,  all $d_{n,k}$ are bounded and hence $d_{n,k}(\Psi')^{k}$ belong to higher order error terms. According to the difference of vanishing speed, we may rewrite equation \eqref{equ-temp6} into
\begin{equation}\label{equ-temp7}
\begin{array}{lll}
&\displaystyle  \left(\dfrac{n+1}{2n}\right)^{n+1}\\
=&
\displaystyle\frac{n-1}{2n}c_n^{n+1}+\left(\frac{n-1}{2n}c_n^n+c_n^{n-1}  s\Psi''\right)s^{-1}\Psi
+c_n^{n}s\Psi''\\
&\displaystyle +\left(\dfrac{n-1}{2}c_n^n+(\Psi\cdot\Psi''+c_ns\Psi'')\cdot (n-1)c_n^{n-2}+ \frac{n-1}{2}c_n^{n-1}s^{-1}\Psi\right)\Psi'
\\
&\displaystyle+\sum_{k=2}^{n}
d_{n,k}(\Psi')^k\\
=:&\displaystyle
\frac{n-1}{2n}c_n^{n+1}+\left(\frac{n-1}{2n}c_n^n+c_n^nR_2(s)\right)s^{-1}\Psi
+c_n^{n}s\Psi''+\left(\dfrac{n-1}{2}c_n^n+R_1(s)\right)\Psi',
\end{array}
\end{equation}
where $R_1$ and $R_2$ are smooth functions given by
\[R_1(s):=(\Psi\cdot\Psi''+c_ns\Psi'')\cdot (n-1)c_n^{-2}+\frac{n-1}{2}c_n^{-1}+c_n^{-n}\sum_{k=2}^nd_{n,k}(\Psi')^{k-1},\quad  R_2(s):=c_n^{-1}s\Psi''.
\]
Especially, $R_1$ and $R_2$ satisfies
\begin{equation}\label{equ-temp3}
R_1(s),~R_2(s)=O(s^{-1}|\Psi(s)|)+O(|\Psi'(s)|)+O(s|\Psi''(s)|),\quad\text{as }s\rightarrow\infty.
\end{equation}
From the definition of $c_n$, we have
\[
\frac{n-1}{2n}c_n^{n+1}=\left(\dfrac{n+1}{2n}\right)^{n+1}
\]
and hence \eqref{equ-temp7} implies that  for  large $s$,
\[
c_n^{n}s\Psi''+
\left(\frac{n-1}{2}c_n^n+c_n^nR_1(s)\right)\Psi'+
\left(\frac{n-1}{2n}c_n^n+c_n^nR_2(s)\right)s^{-1}\Psi=0,
\]
i.e.,
\[
s^2\Psi''+\left(\frac{n-1}{2}+R_1(s)\right)s\Psi'+\left(\frac{n-1}{2n}+R_1(s)\right)\Psi=0.
\]

In order to apply asymptotic stability theories, we set
\[
\tilde\Psi(\tilde t):=\Psi(s(\tilde t)),\quad \tilde t:=\ln s.
\]
Then
\[
\tilde\Psi'(\tilde t)=\Psi'(s)\cdot e^{\tilde t}=s\Psi'\quad\text{and}\quad
\tilde\Psi''(\tilde t)=\Psi''(s)\cdot e^{2\tilde t}+\Psi'(s)\cdot e^{\tilde t}=s^2\Psi''+s\Psi',
\]
hence $\tilde\Psi$
satisfies
\[
\left\{
\begin{array}{llll}
\displaystyle\tilde\Psi''+\left(\frac{n-3}{2}+\tilde R_1(\tilde t)\right)\tilde\Psi'+
\left(\frac{n-1}{2n}+\tilde R_2(\tilde t)\right)\tilde\Psi=0,\\
\tilde R_1(\tilde t):=R_1(e^{\tilde t})=o(1),\quad R_2(\tilde t):=R_2(e^{\tilde t})=o(1),\quad\text{as }\tilde t\rightarrow\infty.
\end{array}
\right.
\]
The limiting characteristic polynomial
\[
\lambda^2+\frac{n-3}{2}\lambda+\frac{n-1}{2n}=0
\]
admits two roots
\[
\lambda_{\pm}:=\dfrac{-\dfrac{n-3}{2}\pm\sqrt{\dfrac{(n-3)^2}{4}-\dfrac{2(n-1)}{n}}}{2}.
\]
When $n\geq 4$, the real parts of the characteristic roots are negative.
Hence by the asymptotic stability result as in \cite[Chap 13, Theorem 1.1]{Book-Coddington-Levinson-ODE}, there exists a dimensional constant $k_n'\in [\mathtt{Re}(\lambda_+),0)$  such that
\[
|\tilde\Psi(\tilde t)|\leq Ce^{k_n\tilde t}
\]
for sufficiently large $\tilde t$ and $C$.

Especially by \eqref{equ-temp3} and the asymptotics of $\tilde\Psi$,
\[
\tilde R_1(\tilde t)=
R_1(e^{\tilde t})=O(e^{-\tilde t}e^{k_n'\tilde t})\quad\text{and}\quad
\tilde R_2(\tilde t)=O(e^{-\tilde t}e^{k_n'\tilde t})
\quad\text{as }\tilde t\rightarrow\infty.
\]
Thus $\tilde R_1$ and $\tilde R_2$ belong to $L^1((0,\infty))$.
When $n\geq 6$, we have
\[
\dfrac{(n-3)^2}{4}-\dfrac{2(n-1)}{n}>0
\]
and hence $\mathtt{Re}(\lambda_-)<\mathtt{Re}(\lambda_+)$.
When $n=4$ or $5$, $\dfrac{(n-3)^2}{4}-\dfrac{2(n-1)}{n}<0$ and hence
\[
\mathtt{Re}(\lambda_-)=\mathtt{Re}(\lambda_+)\quad\text{but}\quad \lambda_-\neq\lambda_+.
\]
By the asymptotic behavior results in
 \cite[Sec 8.2. (i) and (iii)]{Book-Bodine-Lutz-AsymptoticIntegration}, we have the desired asymptotic behavior result and  $k_n$ can be chosen explicitly as the real part of $\lambda_+.$
From the definition of $\Psi$ and $\Phi$,
\[
\Psi(s)=O(s^{k_n})\quad\text{and}\quad\Phi(s)=c_ns+o(s^{k_n})\quad\text{as }s\rightarrow\infty.
\]
This finishes the proof of the refined asymptotic behavior.
\end{proof}

\section{Result when $n=1$}\label{sec-Dim1}

Similar to the computations in section \ref{sec-Existence}, we only need to analyze monotone non-decreasing solution to the initial value problem
\begin{equation}\label{equ-Sys-InteForm-Dim1}
\left\{
\begin{array}{lll}
  \displaystyle \varphi\cdot\varphi''=1, & r>0,\\
  \varphi'(0)=0,\\
  \varphi(0)=1.
\end{array}
\right.
\end{equation}

The existence  of solution to \eqref{equ-Sys-InteForm-Dim1} can be proved similarly by constructing $\epsilon$-approximation solution as in Lemma \ref{Lem-LocalExis}. The uniqueness of   solution to \eqref{equ-Sys-InteForm-Dim1} can be proved   by rewriting the equation in \eqref{equ-Sys-InteForm-Dim1} into
\[
\varphi'(r)=\int_0^r\dfrac{1}{\varphi(s)}\mathrm ds=F(r,\varphi),\quad\forall~r>0,
\]
and similar argument  as in the proof of Theorem \ref{thm-main-0}.

It remains to study the asymptotic behavior of $\varphi$ near infinity.
Since $\varphi''>0$ and $\varphi'(0)=0$, we may change of variable by setting $\varphi'=P(\varphi)$ for some function $P$. By a direct computation,
\[
\varphi''=P'(\varphi)\cdot\varphi'=P'\cdot P.
\]
Thus the equation in \eqref{equ-Sys-InteForm-Dim1} can be rewritten into
\[
\left\{
\begin{array}{llll}
(P^2)'=\frac{2}{\varphi},& \varphi\geq 1,\\
P(1)=0.\\
\end{array}
\right.
\]
Consequently,
\begin{equation}\label{equ-temp8}
\varphi'(r)=P(\varphi(r))=2^{\frac{1}{2}}(\ln\varphi(r))^{\frac{1}{2}},\quad\forall~r>0.
\end{equation}
Rewrite equation \eqref{equ-temp8} into
\[
2^{\frac{1}{2}}=\dfrac{\varphi'(r)}{(\ln\varphi(r))^{\frac{1}{2}}},\quad\forall~r>0.
\]
Integrating the aforementioned equation over $(1,r)$, we have
\[
2^{\frac{1}{2}}r-2^{\frac{1}{2}}=\int_1^r\dfrac{\varphi'(s)}{(\ln\varphi(s))^{\frac{1}{2}}}\mathrm ds,\quad\forall~r>1.
\]
Integral by parts and divede both sides by $r$,
\[
2^{\frac{1}{2}}-2^{\frac{1}{2}}r^{-1}=\dfrac{\varphi(r)}{r (\ln\varphi(r))^{\frac{1}{2}}}-\dfrac{\varphi(1)}{r(\ln\varphi(1))^{\frac{1}{2}}}
+\dfrac{1}{2r}\int_1^r\dfrac{\varphi'(s)}{ (\ln\varphi(s))^{\frac{3}{2}} }\mathrm ds.
\]
Sending $r\rightarrow\infty$, by the L'Hospital's rule, we have
\[
\begin{array}{llll}
&\displaystyle\lim_{r\rightarrow\infty }\dfrac{1}{2r}\int_1^r\dfrac{\varphi'(s)}{(\ln\varphi(s))^{\frac{3}{2}} }\mathrm ds\\
=&\displaystyle
\lim_{r\rightarrow\infty}\dfrac{\varphi'(r)}{2(\ln\varphi(r))^{\frac{3}{2}} }\\
=&\displaystyle
\lim_{r\rightarrow\infty}\dfrac{\varphi''(r)\cdot \varphi(r)}{3(\ln\varphi(r))^{\frac{1}{2}}\cdot\varphi'(r)
}\\
=&0,
\end{array}
\]
where we used the fact that
\[
\varphi\cdot\varphi''=1,\quad\lim_{r\rightarrow\infty}
\dfrac{1}{(\ln \varphi(r))^{\frac{1}{2}}}=0\quad\text{and}\quad
\varphi'(r)\geq \varphi'(1)>0,\quad\forall~r\geq 1.
\]
Combining the identities above,
\[
2^{\frac{1}{2}}=\lim_{r\rightarrow\infty}\dfrac{\varphi(r)}{r(\ln\varphi(r))^{\frac{1}{2}}}.
\]
This finishes the proof of Remark \ref{Rem-Dim1Case}.

\small

\bibliographystyle{plain}

\bibliography{C:/Bib/Thesis}

\bigskip

\noindent N. An \& J. Bao

\medskip

\noindent  School of Mathematical Sciences, Beijing Normal University\\
Laboratory of Mathematics and Complex Systems, Ministry of Education\\
Beijing 100875, China \\[1mm]
Email: \textsf{AnNingAN@mail.bnu.edu.cn, jgbao@bnu.edu.cn}

\medskip

\noindent Z. Liu (Corresponding Author)

\medskip

\noindent  Institute of Applied Mathematics, Department of Mathematics\\
Facility of Science, Beijing University of Technology\\
Beijing 100124, China \\[1mm]
Email:\textsf{liuzixiao@bjut.edu.cn}

\end{document}